\documentclass[12pt]{amsart}
\usepackage{enumerate, amsmath, amsthm, amsfonts, amssymb, mathrsfs, graphicx, paralist, stmaryrd}
\usepackage[all]{xy}
\usepackage[usenames, dvipsnames]{color}
\usepackage[margin=1in]{geometry} 
\usepackage[bookmarks, bookmarksdepth=2, colorlinks=true, linkcolor=blue, citecolor=blue, urlcolor=blue]{hyperref}
\usepackage{enumitem}
\usepackage{tikz-cd}
\usepackage{tikz}
\usepackage{mathtools}

\newcommand{\cC}{\mathcal{C}}

\newcommand{\bN}{\mathbf{N}}

\newcommand{\bS}{\mathbf{S}}
\newcommand{\cS}{\mathcal{S}}
\newcommand{\fS}{\mathfrak{S}}

\newcommand{\bV}{\mathbf{V}}

\newcommand{\bZ}{\mathbf{Z}}

\renewcommand{\phi}{\varphi}

\newcommand{\lw}{{\textstyle \bigwedge}}

\makeatletter
\def\Ddots{\mathinner{\mkern1mu\raise\p@
\vbox{\kern7\p@\hbox{.}}\mkern2mu
\raise4\p@\hbox{.}\mkern2mu\raise7\p@\hbox{.}\mkern1mu}}
\makeatother

\DeclareMathOperator{\triv}{triv}

\DeclareMathOperator{\Sym}{Sym}

\DeclareMathOperator{\GKdim}{GKdim}

\DeclareMathOperator{\sgn}{sgn}

\DeclareMathOperator{\len}{len}

\DeclareMathOperator{\Ind}{Ind}
\DeclareMathOperator{\Hom}{Hom}

\DeclareMathOperator{\Mod}{Mod}
\newcommand{\id}{\mathrm{id}}

\newcommand{\pol}{\mathrm{pol}}

\DeclareMathOperator{\Rep}{Rep}

\DeclareMathOperator{\colim}{colim}

\DeclareMathOperator{\socle}{soc}

\DeclareMathOperator{\chark}{char}

\newcommand{\GL}{\mathbf{GL}}

\newcommand{\SL}{\mathbf{SL}}

\makeatletter
\@addtoreset{equation}{section}
\makeatother

\numberwithin{equation}{section}
\newtheorem{theorem}[equation]{Theorem}

\newtheorem{proposition}[equation]{Proposition}
\newtheorem{lemma}[equation]{Lemma}
\newtheorem{corollary}[equation]{Corollary}

\theoremstyle{definition}
\newtheorem{rmk}[equation]{Remark}
\newenvironment{remark}[1][]{\begin{rmk}[#1] \pushQED{\qed}}{\popQED \end{rmk}}
\newtheorem{eg}[equation]{Example}
\newenvironment{example}[1][]{\begin{eg}[#1] \pushQED{\qed}}{\popQED \end{eg}}
\newtheorem{defn}[equation]{Definition}
\newenvironment{definition}[1][]{\begin{defn}[#1]\pushQED{\qed}}{\popQED \end{defn}}

\makeatletter
\renewcommand{\thesubsection}{%
  \ifnum\c@subsection<1 \@arabic\c@section
  \else \thesection.\@arabic\c@subsection
  \fi
}
\makeatother

\title{Noncommutative invariants of finite and classical groups}
\author{Karthik Ganapathy}
\address{Department of Mathematics, University of California, San Diego, CA}
\email{\href{mailto:kganapathy@ucsd.edu}{kganapathy@ucsd.edu}}
\urladdr{\url{https://sites.google.com/view/karthik-ganapathy}}
\begin{document}
\begin{abstract}
We investigate the structure of the invariant subring of the tensor algebra $T(W)$ of a $G$-representation $W$, viewed as a twisted commutative algebra (tca). For a faithful representation $W$ of a finite group $G$ over a field $k$, we show that if char$(k) \mid \#G$, then $T(W)^G$ is not finitely generated as a tca.  In contrast, for a representation $W$ of a classical group $G_{\mathbb{Z}}$, we prove that the invariant subring $T(W_k)^{G_k}$ is finitely generated as a tca when $k$ is algebraically closed of sufficiently large characteristic, provided that $W$ admits a good filtration over $\mathbb{Z}$.  Finally, we introduce a categorical variant of the Gelfand--Kirillov dimension and compute its value to be $\binom{n+1}{2}$ for $T(\mathbb{C}^n)$ as a tca. Our key insight is to use the Schur functor to reduce questions about noncommutative invariants to those concerning \textit{vector invariants}.
\end{abstract}
\maketitle
\section{Introduction}
An overarching goal of this paper is to demonstrate how Schur--Weyl duality enables us to answer questions in noncommutative ring theory using standard methods from commutative algebra, even in positive characteristic.
We specifically address the finite generation problem for invariants of tensor algebras by relating it to Weyl's polarization technique using the Schur functor. Throughout this paper, we let $k$ be an infinite field of characteristic $p \geq 0$.
\subsection{Noncommutative invariants.} 
Given a finite group $G$ acting on a finite-dimensional vector space $W$, it is well known that the invariant subring $\Sym(W)^G$ is finitely generated. The analogous question for the tensor algebra $T(W)$ is quite subtle with a rich history. Wolf \cite{wol36sym} showed that the ring of noncommutative symmetric polynomials in $\geq 2$ variables is not finitely generated. Almost half a century later, Dicks--Formanek \cite{df82poin} proved that this almost always holds: the subring of noncommutative invariants $T(W)^G$ is a finitely generated $k$-algebra if and only if $G$ acts via a finite cyclic group consisting of scalar matrices. 

Khar{\v c}enko recognized the significance of the commuting actions of the symmetric groups $\fS_n$ on $T(W)_n$ to this question. He \cite{kh84nc} asked whether $T(W)^G$ admits a finite generating set when incorporating this action, noting that this holds for soluble groups $G$ when $p=0$.
Koryukin \cite{kor84nci} answered Khar{\v c}enko's question affirmatively for linearly reductive groups. To that end, he introduced the notion of an ``S-algebra" and proved a noetherian result for $T(W)$ as an S-algebra. From this, he deduced that $T(W)^G$ is also a finitely generated S-algebra using the Reynolds operator. Over the past decade, Sam--Snowden \cite{ss12tca, ss17grobner} and others have independently studied S-algebras under the name ``twisted commutative algebras" (tcas), and we adopt this terminology exclusively. See Section~\ref{ss:commalg} for definitions. 

\subsection{Our results.}
We negatively resolve the finite generation problem for noncommutative invariants of finite groups in the modular setting.
\begin{theorem}\label{thm:finitegps}
Assume $G \subset GL(W)$ is a finite group. 
As a twisted commutative algebra, the invariant subring $T(W)^G$ 
\begin{itemize}
    \item is noetherian and generated by elements of degree $\leq \#G$, if $p \nmid \#G$ or $p = 0$; and 
    \item is not finitely generated, if $p \mid \#G$.
\end{itemize}
\end{theorem}
The infinite generation of the invariant ring for the permutation action of $\fS_n$ on $T(k^n)$ as a tca when $n \geq p > 0$ was obtained only recently in \cite{bddk23sym} (see Example~\ref{exm:multisymmetric}). 
Our contribution to the non-modular setting (the first part) is limited to the following: we establish a module-theoretic noetherian result, whereas, in our terminology, Koryukin only shows that $T(W)^G$ is \textit{weakly noetherian}, and we make his result effective.

Our key idea is to bridge the study of noncommutative invariants with the invariants of certain commutative---albeit infinite-dimensional---algebras using Schur--Weyl duality and then invoke recent results from modular invariant theory. 
This setup is explained in detail in Section~\ref{s:schur}, but we sketch the idea here. Let $S = \Sym(W \otimes \bV)$ and $R = T(W)$ where $W$ is a representation of a group $G$ and $\bV \coloneqq \colim k^n$ is an infinite-dimensional vector space over $k$. The objects $S$ and $R$ can be viewed as free commutative algebra objects in two important categories from representation theory, namely $\Rep^{\pol}(\GL)$ and $\Rep(\fS_*)$.
Furthermore, $S$ is mapped to $R$ under the Schur functor $\cS \colon \Rep^{\pol}(\GL) \to \Rep(\fS_*)$. Our key observation is:
\begin{theorem}\label{thm:structure}
Under the identification $\cS(S) \cong R$ above, 
    \begin{itemize}
        \item we have an isomorphism $\cS(S^G) \cong R^G$, and
        \item given a $\GL$-subrepresentation $U \subset S^G$ that generates $S^G$, the subspace $\cS(U) \subset R^G$ generates $R^G$.
    \end{itemize}
\end{theorem}
In Section~\ref{s:proofs}, we first prove Theorem~\ref{thm:structure}. The second part of Theorem~\ref{thm:finitegps} now follows by transporting results of Richman \cite{rich96inv} regarding the failure of Weyl's polarization theorem for finite groups in the modular setting. Similarly, the first part of Theorem~\ref{thm:finitegps} follows by invoking Fleischmann's work, which extends Noether's bound to all characteristics \cite{flei00noeth}, combined with standard arguments involving the Reynolds operator.

While the first part of our main theorem extends easily to infinite linearly reductive groups, the second part does not immediately generalize to infinite reductive groups. We partially address this with our next result.
Let $G$ be a classical group and $W$ be a finite $\bZ$-module with a linear action of $G$ affording a universal good filtration; see Section~\ref{ss:goodfiltration}. The base change $W_k = W \otimes k$ is a representation of $G_k = G \times_{\bZ} k$. By carrying over results of Derksen--Makam \cite{dm20weyl} through the bridge from Theorem~\ref{thm:structure}, we obtain:
\begin{theorem}\label{thm:chevalley}
There exists $N > 0$ such that for all algebraically closed fields $k$ of characteristic $p > N$, the tca $T(W_k)^{G_k}$ is finitely generated.
\end{theorem}
The conclusion of the above theorem can fail for small $p > 0$ (Example~\ref{exm:matrixinv}).

In the final section, we introduce the Gelfand--Kirillov dimension for finitely generated algebra objects in a locally finite length symmetric monoidal $k$-linear category and compute it for some tcas. Notably, in characteristic zero, $\GKdim$ of polynomial tcas generated in degree one agrees with the Krull--Gabriel dimension of their module category as computed by Nagpal--Sam--Snowden \cite[Section~3.5]{ss19gl2}. It would be interesting to prove a general result along these lines.
\subsection{Related work} \label{ss:related}
Tambour and Teranishi independently pioneered the use of Schur--Weyl duality to study noncommutative invariants, albeit only in characteristic zero.
In \cite{ter88nato}, Teranishi explains how to explicitly obtain generators of $T(W)^G$ as a tca using generators of $\Sym(W^{\oplus \dim W})^G$, a method made possible by combining Schur--Weyl duality with Weyl's polarization theorem. He \cite{ter91yd} also provides a novel method to derive generators of $T(W)^G$ from those of $\Sym(W \oplus \lw^2 W)^G$. In his thesis \cite{tam92salg}, Tambour reproves Koryukin's theorem, obtains degree bounds for the invariant ring, and introduces and proves rationality results for a certain enhanced Hilbert series by using classical results from invariant theory like Noether's bound and Molien's formula.

The structure of the invariant rings occurring in Theorem~\ref{thm:chevalley} is of great interest, as determining whether or not they are noetherian is a major open problem in the field of representation stability (see \cite{gan24non, gan25non} for related results). While we do not address this in the present paper, we prove some auxiliary results (Theorem~\ref{thm:specht}).

Twisted (commutative) algebras, even infinitely generated ones, arise naturally when studying the topology of configuration spaces of points/disks on a plane (see, for example, \cite{am24conf, waw23thesis} and references therein). We hope the method developed in Section~\ref{s:dimensions} can be used to provide interesting invariants of topological spaces.

\subsection*{Acknowledgments.}{The author thanks M{\' a}ty{\' a}s Domokos, Nate Harman, Martin Kassabov, and Andrew Snowden for helpful discussions over email and is especially grateful to Lucas Buzaglo for patiently explaining foundational results about noncommutative rings.}

\section{Setup} \label{s:schur}
Throughout this paper, a \textbf{partition} $\lambda = (\lambda_1, \lambda_2, \ldots)$ is an infinite sequence of integers with $\lambda_i \geq \lambda_{i+1}$ for all $i \geq 1$ and $\lambda_j = 0$ for $j \gg 0$. The \textbf{size} of the partition $|\lambda| = \sum_i \lambda_i$. A partition is \textbf{$p$-restricted} if $\lambda_i - \lambda_{i+1} < p$ for all $i \geq 1$. We also let $\bV$ be the infinite-dimensional $k$-vector space with basis $e_1, e_2, \ldots,$ and $\GL$ be the group of $k$-linear automorphisms of $\bV$.

\subsection{Symmetric monoidal categories} \label{ss:symmmon}
We first sketch the definition of a symmetric monoidal $k$-linear category and functors between them. 
A \textbf{symmetric monoidal $k$-linear abelian category} is the data $(M, \otimes, 1, \alpha, \lambda, B)$ where:
\begin{enumerate}
  \item $M$ is a {$k$-linear abelian category}
  \item $\otimes$ is a $k$-bilinear functor
  \[
  \otimes : M \times M \to M
  \]
  called the {tensor product};
  \item $1 \in M$
  is an object called the {unit object}; and
  \item  we have natural isomorphisms
  \begin{align*}
    a_{x,y,z} &: (x \otimes y) \otimes z \to x \otimes (y \otimes z) & \text{(associator)} \\
    \lambda_x &: 1 \otimes x \to x & \text{(left unitor)} \\
    % \rho_x &: x \otimes 1 \to x && \text{(right unitor)} \\
    B_{x,y} &: x \otimes y \to y \otimes x & \text{(braiding)}
  \end{align*}
\end{enumerate}
such that $B_{x, y} B_{y, x} = \id_{y \otimes x}$ and certain coherence axioms are satisfied (the associator obeys the pentagonal identity; the associator and unitor obey the triangle identity; and the braiding and associator obey the hexagonal identity). For the full coherence axioms, 
we refer the reader to \cite[Chapter~1.3]{reischuk16thesis}.

Given two symmetric monoidal $k$-linear abelian categories $(M, \otimes_M, 1_M, a^M, \lambda^M, B^M)$ and $(N, \otimes_N, 1_N, a^N, \lambda^N, B^N)$, a \textbf{symmetric monoidal $k$-linear functor}  $M \to N$ is the data $(F, \phi, \epsilon)$ where
\begin{enumerate}
  \item $F \colon M \to N$ is a $k$-linear functor;
  \item $\phi$ is a natural isomorphism
  \[
  \phi_{x,y}: F(x) \otimes_N F(y) \xrightarrow{\sim} F(x \otimes_M y), \quad \forall x,y \in M,
  \]
  \item $\epsilon$ is an isomorphism $1_N \xrightarrow{\sim} F(1_M)$, 
\end{enumerate}
and these morphisms are also required to satisfy certain compatibility conditions (with the associativity, unitor, and braiding morphisms).

For the remainder of the paper, we will suppress the natural maps and the tensor product when referring to a symmetric monoidal $k$-linear category. Given a $k$-linear functor $F$ between two symmetric monoidal $k$-linear categories, we say it is \textit{symmetric monoidal} if there exists an isomorphism $\epsilon$ and a natural isomorphism $\phi$ such that $(F, \phi, \epsilon)$ is a symmetric monoidal $k$-linear functor as defined above.

\subsection{Representation categories}
We introduce the two symmetric monoidal $k$-linear categories that are central to this paper. 
\begin{enumerate}[wide=1cm]
\item  $\Rep^{\pol}(\GL):$ A representation of $\GL$ is \textbf{polynomial} if it occurs as a subquotient of direct sums of $\bV^{\otimes n}$, with $n$ allowed to vary. We let $\Rep^{\pol}(\GL)$ denote the category of polynomial representations of $\GL$; it is a symmetric monoidal category with the usual tensor product of representations (denoted $\otimes$) and the usual braiding on vector spaces. 

The simple objects in $\Rep^{\pol}(\GL)$ are indexed by arbitrary partitions $\lambda$ (corresponding to their highest weight) and denoted $L_{\lambda}(\bV)$. Of particular interest are also the Schur modules, denoted $\bS_{\lambda}(\bV)$; the simple module $L_{\lambda}(\bV)$ is the socle of $\bS_{\lambda}(\bV)$. 
\item $\Rep(\fS_*)$: An $\fS_*$-representation is a sequence $\{V_n\}_{n \in \bN}$ where each $V_n$ is a representation of the symmetric group $\fS_n$ over $k$. We let $\Rep(\fS_*)$ be the category of $\fS_*$-representations.
Given objects $V \coloneqq \{V_n\}$ and $W \coloneqq \{W_n\}$, we define 
\[(V \boxtimes W)_n = \bigoplus_{i=0}^n \Ind_{\fS_i \times \fS_{n-i}}^{\fS_{n}} V_i \otimes W_{n-i},\] 
where $V_i \otimes W_{n-i}$ is the external tensor product of the $\fS_i$-representation and the $\fS_{n-i}$-representation $n-i$. 
This endows $\Rep(\fS_*)$ with the structure of a symmetric monoidal category 
where the braiding is the one induced from the braiding on vector spaces. 

The simple objects in $\Rep(\fS_*)$ are indexed by $p$-restricted partitions $\lambda$ and denoted $D_{\lambda}$. We will also be interested in the Specht modules $S^{\lambda}$ which are the base changes of the integrally defined Specht modules via $\bZ \to k$. Similar to the $\GL$-case, we have $D_{\lambda} = \socle(S^{\lambda})$. 
\end{enumerate}
We caution that the standard indexing of irreducible representations of the symmetric groups is by ``$p$-regular partitions" and denoted $D^{\lambda}$; this is conjugate to our indexing, and in fact, $D^{\lambda} = D_{\lambda^t} \otimes \sgn$, where $\sgn$ is the sign representation of $\fS_{|\lambda|}$. This subtlety is not relevant to the contents of our paper; the interested reader may refer to \cite[\S6.4]{green07pol}.

Both categories have an \textbf{$\bN$-grading} where $L_{\lambda}(\bV)$ and $D_{\lambda}$ are homogeneous of degree $|\lambda|$; the monoidal structure is compatible with this grading.

\subsection{Commutative algebra objects}\label{ss:commalg}
Given a symmetric monoidal $k$-linear category $\cC$, a \textbf{commutative algebra object} $(A, \mu, \eta)$ is an object $A$ in $\cC$ with a unit map $\eta \colon 1_{\cC} \to A$ and a multiplication map $\mu \colon A \otimes A \to A$ such that the following diagrams commute (corresponding respectively to the associativity, unit, and commutativity axioms):
\[
\begin{tikzcd}[row sep=large, column sep=large]
(A \otimes A) \otimes A \ar[r, "a_{A,A,A}"] \ar[d, "\mu \otimes 1_A"'] & A \otimes (A \otimes A) \ar[d, "1_A \otimes \mu"] \\
A \otimes A \ar[r, "\mu"'] & A
\end{tikzcd}
\quad
\begin{tikzcd}[row sep=large, column sep=large]
1 \otimes A \ar[r, "\lambda_A"] \ar[d, "\eta \otimes 1_A"'] & A \\
A \otimes A \ar[ur, "\mu"'] &
\end{tikzcd}
\quad 
\begin{tikzcd}[row sep=large, column sep=large]
A \otimes A \ar[r, "B_{A,A}"] \ar[dr, "\mu"'] & A \otimes A \ar[d, "\mu"] \\
& A
\end{tikzcd}
\]
We will, by slight abuse of notation, refer to a commutative algebra object simply by its underlying object. Similar to the above discussion, given a commutative algebra object $A$, one can also define the notion of a \textbf{module object} for $A$ in $\cC$.

We are now prepared to define our main objects of study.

\begin{definition}\label{defn:glalg}
A {\bf $\GL$-algebra} is a commutative algebra object in $\Rep^{\pol}(\GL)$. 
\end{definition}
Concretely, a $\GL$-algebra $S$ is an (ordinary) commutative $k$-algebra $S$ with an action of $\GL$ by $k$-algebra automorphisms under which $S$ forms a polynomial representation of $\GL$. 

The prototypical example of a $\GL$-algebra is the polynomial ring $\Sym(U)$ where $U$ is a polynomial representation of $\GL$. In this paper, we will mostly be concerned with the $\GL$-algebra $\Sym(W \otimes \bV)$ where $W$ is a finite-dimensional $k$-vector space with trivial $\GL$-action.

\begin{definition}\label{defn:tca}
A {\bf twisted commutative algebra} (tca) is a commutative algebra object in $\Rep(\fS_*)$. 
\end{definition}
Concretely, a tca is an associative, unital, $\bN$-graded $k$-algebra $A=\bigoplus_{n \ge 0} A_n$ equipped with an action of $\fS_n$ on $A_n$ such that:
\begin{enumerate}
\item the multiplication map $A_n \otimes A_m \to A_{n+m}$ is $(\fS_n \times \fS_m)$-equivariant (we use the standard embedding $\fS_n \times \fS_m \subset \fS_{n+m}$ for the action on $A_{n+m}$); and 
\item given $x \in A_n$ and $y \in A_m$ we have $xy=\tau_{m, n}(yx)$, where $\tau_{m,n} \in \fS_{n+m}$ is defined by
\begin{displaymath}
\tau(i) = \begin{cases} i+n & \text{if } 1 \le i \le m, \\ i-m & \text{if } m+1 \le i \le n+m. \end{cases}
\end{displaymath}
\end{enumerate} 
The last condition above is the ``twisted commutativity'' condition.

Let $W$ be a $k$-vector space considered as a representation of $\fS_1$; we emphasize this by instead writing $W\langle 1 \rangle$. The free tca over $W\langle 1 \rangle $ is the non-commutative polynomial ring $T(W)$ over $W$ with the $S_n$-action on $T(W)_n$ being the obvious one: on monomials, it is given by $\sigma(w_{1} \cdots w_{n})=w_{{\sigma^{-1}(1)}} \cdots w_{{\sigma^{-1}(n)}}$. 

\subsection{Finite generation} \label{ss:fgen}
Let $S$ be a $\GL$-algebra and $R$ be a tca.
The $\GL$-algebra $S$ is \textbf{finitely generated in degrees $\leq s$} if it is generated (as a $k$-algebra) by the $\GL$-orbit of finitely many homogeneous elements all of degree $\leq s$. Similarly, the tca $R$ is \textbf{finitely generated in degrees $\leq s$} if there exist finitely many elements $r_i$ in various $R_{n_i}$ with $n_i \leq s$, such that the smallest subspace of $R$ containing the $r_i$ and stable under multiplication and the action of the symmetric groups is $R$ itself. The $\GL$-algebra $\Sym(U)$ (resp.~tca $\Sym(\{V_n\})$) is finitely generated if and only if $U$ is a finite-length polynomial representation (resp.~$\sum_n \dim(V_n)$ is finite). 

We also consider module objects for $S$ in $\Rep^{\pol}(\GL)$; this forms an abelian category $\Mod_S$. The $\GL$-algebra $S$ is \textbf{weakly noetherian} if $S$-submodules of $S$ satisfy the ascending chain condition, and it is \textbf{noetherian} if $\Mod_S$ is locally noetherian. These definitions extend analogously to the tca $R$ (and in fact, for any commutative algebra object in a symmetric monoidal $k$-linear category).

\subsection{The Schur functor} \label{ss:schur}
We have introduced two parallel concepts that, despite their apparent differences, are closely connected through the Schur functor.
The \textbf{Schur functor} $\cS \colon \Rep^{\pol}(\GL) \to \Rep(\fS_*)$ is defined on objects by the formula 
\[\cS(U) \coloneqq \{ \Hom_{\GL}(\bV^{\otimes n}, U) \}_{n \in \bN}\]
which can be identified with the $(1^n)$ weight space of $U$ with $n$ varying
(see \cite[\S~6.1]{green07pol}). 
The Schur functor is an equivalence of symmetric monoidal $k$-linear categories when $p = 0$. When $p > 0$, it is a $k$-linear functor that still satisfies many pleasant properties.
\begin{lemma}\label{lem:schurfacts}
Assume $p > 0$.
\begin{enumerate}
    \item The Schur functor $\cS$ is exact.
    \item For a partition $\lambda$, we have $\cS(L_{\lambda}(\bV)) \ne 0$ if and only if $\lambda$ is $p$-restricted.
    \item For all partitions $\lambda$, we have $\cS(\bS_{\lambda}(\bV)) = S^{\lambda}$.
\end{enumerate}
\end{lemma}
\begin{proof}
    \begin{inparaenum}
        \item is \cite[6.2a]{green07pol}.
        \item is obtained by combining \cite[6.4a]{green07pol} and \cite[6.4b]{green07pol}.
        \item is \cite[6.3(e)]{green07pol}.
    \end{inparaenum}\end{proof}

We sketch the proof of the next well-known result, as it plays a crucial role in our method. 
\begin{proposition}\label{prop:schurmon}
    The Schur functor $\cS$ is symmetric monoidal.
\end{proposition}
\begin{proof}
 We first define a natural map $\phi_{U, W} \colon \cS(U) \boxtimes \cS(W) \to \cS(U \otimes W)$. Without loss of generality, we may assume $U$ and $W$ are of finite length and homogeneous of degree $d$ and $e$, respectively. Recall that $\Hom_{\GL}(\bV^{\otimes n}, Z)$ is naturally isomorphic to the $(1^n)$ weight space of $Z$ with the isomorphism given by $v \in Z_{(1^n)} \mapsto f_v$, where $f_v$ is the $\GL$-equivariant map which maps 
   \[e_1 \otimes e_2 \otimes \ldots \otimes e_n \in \bV^{\otimes n} \text{ to } v \in Z.\]
    We define an $\fS_d \times \fS_e$-equivariant map from $\cS(U) \times \cS(W) \to \cS(U \otimes W)$ where $f_u \otimes f_w$ with $f_u \in \Hom_{\GL}(\bV^{\otimes d}, U)$ and $f_w \in \Hom_{\GL}(\bV^{\otimes e}, W)$ is mapped to $f_{u \otimes w}$. This extends to an $\fS_{d+e}$-equivariant map $\cS(U) \boxtimes \cS(W) \to \cS(U \otimes W)$ and is natural in both variables. 
Furthermore, the domain and codomain (for fixed $U$ and $W$) both have dimension 
    $\binom{d+e}{d}\dim(U_{1^d})\dim(W_{1^e})$. 
By comparing weight spaces, we see that the map is surjective, so it is an isomorphism. This proves that the functor is monoidal, and $\phi$ clearly commutes with the symmetric structures, so $\cS$ is a symmetric monoidal functor as well.
\end{proof}
This proposition implies that the Schur functor takes $\GL$-algebras to twisted commutative algebras. We will mostly use the following concrete application.
\begin{example}
    Let $W$ be a finite-dimensional vector space. By Proposition~\ref{prop:schurmon}, we have $\cS(\Sym(W \otimes \bV)) \cong \Sym(\cS(W \otimes \bV)) = \Sym (W\langle 1 \rangle) = T(W)$.
\end{example}

\section{Finite group invariants}\label{s:proofs}
Throughout this section, we let $W$ be a representation of a group $G$, and $S$ and $R$ be the $\GL$-algebra $\Sym(W \otimes \bV)$ and tca $T(W)$, respectively.
\subsection{The tca structure on \texorpdfstring{$T(W)^G$}{the invariant subring}}
We briefly explain why $T(W)^G$ has the structure of a twisted commutative algebra. The tensor algebra $T(W)$ can be viewed as a tca with $W$ an $\fS_1$-representation. This is the free tca over $W$ as explained after Definition~\ref{defn:tca} so the action of $G$ on $W$ extends to a tca automorphism of $T(W)$ and thereby, the tca structure descends to $T(W)^G$.

\begin{remark}[Freeness of the invariant ring]
For finite groups $G$, the invariant subring $\Sym(W)^G$ is a polynomial ring only if the action of $G$ is generated by pseudo-reflections.
Remarkably, Khar{\v c}enko \cite{kha78free} proved that $T(W)^G$ is always free as an associative $k$-algebra. 
However, the invariant ring $T(W)^G$ is never free as a tca unless $T(W)^G = T(W)$, as can be seen using growth rates:
$T(W)^G_n$ grows exponentially with $n$ because $\dim(T(W)^G_n) \leq \dim(W)^n$, but a free tca with at least one generator of degree $\geq 2$ grows superexponentially. 
\end{remark}

For the next result, we freely use the fact that the Schur functor $\cS$ applied to an object $A$ can be identified with the subspace of $A$ spanned by all vectors of weight $(1^n)$ for all $n \geq 0$ (see Section~\ref{ss:schur}). In what follows, a flat weight is a weight of the form $(1^n)$ for some $n \geq 1$, and a flat weight vector is a weight vector of flat weight.
\begin{proof}[Proof of Theorem~\ref{thm:structure}]
We identify $R$ with the subspace of $S$ spanned by all flat weight vectors. Under this identification, the action of $G$ on $R$ is the induced action on this subspace. Similarly, we identify $\cS(S^G)$ with the subspace of $S^G$ spanned by all flat weight vectors. An element is in $\cS(S^G)$ if and only if it (1) is a sum of flat weight vectors and (2) is also invariant under the action of $G$. By (1), the element is in $R$ and by (2), the element is in $R^G$. The reverse containment is similar.

For the second part, given $U \subset S^G$ that generates $S^G$, we obtain a surjection $\Sym(U) \to S^G$. Applying the Schur functor $\cS$ and using Proposition~\ref{prop:schurmon} and Lemma~\ref{lem:schurfacts}(1), we get a surjection $\Sym(\cS(U)) \to \cS(S^G) \cong R^G$, as required.
\end{proof}

\begin{corollary}\label{cor:glfgtcafg}
Assume $S^G$ is finitely generated as a $\GL$-algebra in degrees $\leq r$. Then $R^G$ is also finitely generated as a twisted commutative algebra in degrees $\leq r$.
\end{corollary}
\begin{proof}
This easily follows from Theorem~\ref{thm:structure} and the fact that the Schur functor takes finite length objects to finite length objects.
\end{proof}

For a $\GL$-algebra or tca $A$, we let $A_{+}$ be the ideal of positive degree elements.

\begin{lemma} \label{lem:p-rest}
Assume for all $n \in \bN$, there exists a $p$-restricted partition $\lambda$ with $|\lambda| > n$ such that $L_{\lambda}(\bV)$ is an irreducible constituent of $S^G_+ / (S^G_+)^2$. Then $R^G$ is not a finitely generated tca.
\end{lemma}
\begin{proof}
By Theorem~\ref{thm:structure}, we have an isomorphism $\cS(S^G_+) \cong R^G_+$. Furthermore, $(S^G_+)^2$ is the image of the multiplication map $S^G_+ \otimes S^G_+ \to S^G_+$; the same holds for $(R^G_+)^2$. Using the fact that $\cS$ is a symmetric monoidal functor, we thus see that $\cS((S^G_+)^2) \cong (R^G_+)^2$, and by the exactness of $\cS$, we also see that $R^G_+/(R^G_+)^2 \cong \cS(S^G/(S^G_+)^2)$.
By Lemma~\ref{lem:schurfacts}(ii), we have $\cS(L_{\lambda}(\bV)) \ne 0$ for $p$-restricted $\lambda$. Therefore, if $S^G_+ / (S^G_+)^2$ has a non $p$-restricted component in arbitrarily large degrees, we have that $R^G_+/(R^G_+)^2$ is also nonzero in infinitely many degrees or, $R^G$ is not finitely generated.
\end{proof}
When $p > 0$, we do not know of an example where $R^G$ is finitely generated but $S^G$ is not.
By the above result, all irreducible constituents of $S^G_+/(S^G_+)^2$ of sufficiently large degree must not be $p$-restricted.

\subsection{Finite groups}
In this subsection, we assume $G$ is a finite group.

The key result that enables us to prove Theorem~\ref{thm:finitegps} is a mild refinement of Richman's remarkable work \cite{rich96inv} regarding the failure of Weyl's polarization theorem for finite groups in the modular setting. Richman shows that if $\chark(k)\mid\#G$, then $S^G$ contains generators of sufficiently large degree -- this result is insufficient for our purpose as the Schur functor kills elements of large multidegree. We begin with a preliminary result.
\begin{lemma}
    Assume $U$ is a polynomial representation of $\GL$ such that none of its irreducible constituents are $p$-restricted. If $w$ is a nonzero weight vector in $U$, the weight of $w$ has at least one component that is $\geq p$. 
\end{lemma}
\begin{proof}
    It suffices to prove this result when $U$ itself is irreducible. Then, by the Steinberg tensor product theorem, we obtain an isomorphism $U \cong L_{\lambda} \otimes L_{\mu}^{(1)}$ where $\lambda$ is $p$-restricted and $\mu$ is nonempty, and $(-)^{(1)}$ is the functor that pulls back a representation along the Frobenius map $\GL \to \GL$. Every weight occurring in $U$ is therefore the sum of weights $\widetilde{\lambda} + p \widetilde{\mu}$ where $\widetilde{\lambda}$ is a weight occurring in $L_{\lambda}$, and similarly for $\widetilde{\mu}$. The result follows.
\end{proof}

The next result follows essentially from Richman's work; we spell this out in the proof.
\begin{proposition}\label{prop:richman}
Assume $G \subset GL(U)$ is a finite group with $\#G$ divisible by $p$. 
For all $n \in \bN$, there exists a $p$-restricted partition $\lambda$ with $|\lambda| > n$ such that $L_{\lambda}(\bV)$ is an irreducible constituent of $S^G_+/(S^G_+)^2$.
\end{proposition}
\begin{proof}
Choose $m \geq \frac{n p^{\#G - 1} - 1}{p - 1}$.
By \cite[Proposition~9]{rich96inv} there exists a nonzero homogeneous element in $S^G$, which we shall call $g$, of degree $\geq m \frac{p-1}{p^{\#G-1}-1}$ such that $g$ is nonzero in $S^G_+/(S^G_+)^2$. We may assume that $g$ is a weight vector. In the proof of \cite[Proposition~7]{rich96inv}, Richman shows that the element $g$ can be chosen such that it contains a monomial that divides a nonzero monomial in the polynomial $f(W, m')$ defined at the end of \cite[Page~30]{rich96inv} for some $m' \leq m$. The key point is that by definition, the polynomial $f(W, m')$ is readily seen to be a weight vector for the $\GL$-action, with all nonzero components of its weight being $p-1$. Therefore, all components of the weight of $g$ are also $< p$. By the previous lemma, we get that the $\GL$-representation generated by $g$ in $S^G_+/(S^G_+)^2$ must contain an irreducible representation that is $p$-restricted.
\end{proof}

We can now prove the theorem on finite groups.
\begin{proof}[Proof of Theorem~\ref{thm:finitegps}]
The tca $T(W)$ is noetherian \cite{ss17grobner}. If $\#G \ne 0$ in $k$, then $T(W)^G$ is a direct summand of $T(W)$ using the Reynolds operator. Therefore, $T(W)^G$ is also noetherian. To obtain the degree bound for $T(W)^G$, note that $\Sym(W \otimes \bV)^G$ is generated in degrees $\leq |G|$ by \cite{flei00noeth}, and therefore so is $T(W)^G$ by Corollary~\ref{cor:glfgtcafg}.

In the case when $p \mid \#G$, the invariant ring $T(W)^G$ is not finitely generated by combining Proposition~\ref{prop:richman} and Lemma~\ref{lem:p-rest}. \end{proof}

In the next example, we explain how our ideas can be used to obtain explicit generators of invariant subrings, recovering results of \cite{bddk22sym, bddk23sym}.
\begin{example}\label{exm:multisymmetric}
   Let $k$ be a field of characteristic $p > 0$. Consider the $\fS_p$ action on $T(k^p) = k\langle x_1, x_2, \ldots, x_p \rangle$. We identify $S = \Sym(k^n \otimes \bV)$ with $k[x_{i, j} | i \in [p], j \in \bN]$. The invariant subring $S^G$ is not finitely generated \cite{ryd07gens}: the polynomials $f_n \coloneqq x_{1,1}x_{1,2}\ldots x_{1,n} + x_{2,1}x_{2,2}\ldots x_{2,n} + \ldots + x_{p,1}x_{p,2}\ldots x_{p,n}$ are nonzero in $S^G_+/(S^G_+)^2$. Furthermore, up to scalars, the polynomials $f_n$ are the only nonzero elements of multidegree $(1^n)$ in $S^G/(S^G_+)^2$. The Schur functor maps the $\GL$-representation generated by the $f_n$ to the subspace generated by $p_n \coloneqq x_1^n + x_2^n + \ldots + x_p^n \in T(k^p)$.
   So $\{p_n\}_n$ minimally generates $R^G$.
\end{example}

\section{Some results for invariants of classical groups} \label{s:chevalley}
In this section alone, we further assume that our field $k$ is algebraically closed of characteristic $p > 0$. Our exposition mostly follows \cite{dm20weyl}. 

Let $G$ be a connected split reductive group scheme over $\bZ$. Let $W$ be a finite $\bZ$-module with an algebraic action of $G$, i.e., $W$ is a comodule over the Hopf algebra $\bZ[G]$. The base change $W_k$ is a rational representation of $G_k = G \times_{\bZ} k$. We prove some results about the structure of $T(W_k)^{G_k}$ as tca for specific representations $W$.

\subsection{Universal good filtration} \label{ss:goodfiltration}
We briefly recall the notion of a good filtration for representations of $G$ over $\bZ$, assuming familiarity with the same notion over algebraically closed fields. A good reference is \cite[Part II]{jan03alg}, especially Appendix~B.

We fix a Borel subgroup $B \subset G$ containing a maximal split torus $T$ of $G$. 
For a dominant integral weight $\lambda$, there exists a $G$-representation $\nabla(\lambda)$ such that the dual Weyl module of $G_k$ corresponding to $\lambda$ is isomorphic to the dual Weyl module $\nabla(\lambda)_k$ for every algebraically closed field $k$. We set $\nabla(\lambda)$ to be the dual Weyl module of $G$ corresponding to $\lambda$ over $\bZ$. A $G$-representation has a \textbf{universal good filtration} if it has a filtration such that the successive quotients are isomorphic to dual Weyl modules $\nabla(\lambda)$ with $\lambda$ allowed to vary. Every $G$-representation that admits a universal good filtration is a free $\bZ$-module since the dual Weyl modules over $\bZ$ are themselves free $\bZ$-modules.

\subsection{Results}\label{ss:chevalleyproof}
We now prove a result that easily implies Theorem~\ref{thm:chevalley} since every classical group is connected and split reductive over $\bZ$.
\begin{theorem}
Assume $W$ is a finite $\bZ$-module with an algebraic action of $G$ that admits a universal good filtration. There exists $N>0$ such that the invariant ring $T(W_k)^{G_k}$ is a finitely generated tca over $k$ for all algebraically closed fields $k$ of char $p > N$.
\end{theorem}
\begin{proof}
By \cite[Theorem~1.7(2)]{dm20weyl} (Theorem~1.10(2) in the arXiv version), the $\GL$-algebra $\Sym(W_k \otimes \bV)^{G_k}$ is finitely generated for $p \gg 0$, and so $T(W_k)^{G_k}$ is a finitely generated tca over $k$ by Corollary~\ref{cor:glfgtcafg}.
\end{proof}

\begin{remark}[Small characteristics] \label{exm:matrixinv}
Consider the rank $n^2$ free $\bZ$-module $W$ consisting of $n \times n$ square matrices endowed with the conjugation action of $\GL_n$. The representation $W$ has a good filtration over $\bZ$ \cite[Page~63]{dkz02mat}. However, the tca $T(W_k)^{GL_n(k)}$ is not finitely generated if char $k \leq n$. Indeed, in Corollary~3.2 of loc.~cit., the authors show that a minimal generating set of $\Sym(W \otimes \bV)^{GL_n(k)}$ contains certain trace elements of weight $1^m$ for all $m > 0$. The result now follows from Lemma~\ref{lem:p-rest} and Lemma~\ref{lem:schurfacts}(2).
\end{remark}

We now establish a result about the $\fS_*$-representation structure of $T(W_k)^{G_k}$.
\begin{theorem} \label{thm:specht}
   Assume $W$ is a $G$-representation over $\bZ$ with a universal good filtration. For $p > \dim(W)$, the invariant ring $T(W_k)^{G_k}$ has a filtration where the successive quotients are Specht modules.
\end{theorem}
\begin{proof}
   By \cite[Corollary~27]{dm20weyl} (Corollary~6.8 in the arXiv version), when $\chark(k) > \dim(W)$, the $\GL$-algebra $\Sym(W \otimes \bV)^{G_k}$ has a filtration with associated graded pieces isomorphic to Schur modules $\bS_{\lambda}(\bV)$. The result now follows by combining (1) and (3) of Lemma~\ref{lem:schurfacts}.
\end{proof}

\section{Gelfand--Kirillov dimension}\label{s:dimensions}
Recall that $k$ is an infinite field of characteristic $p \geq 0$. Let $\cC$ be a symmetric monoidal $k$-linear category of Krull--Gabriel dimension zero, and $A$ be a finitely generated algebra object in $\cC$ with $V \subset A$ being a finite length subobject of $A$ that generates it. Let $f_{A;V}(n)$ be the length of the image $A_{V, \leq n}$ of the canonical map $\bigoplus_{i \leq n} \Sym^i(V) \to A$. 
\begin{defn}
    The \textbf{categorical Gelfand--Kirillov dimension} of $A$ with respect to $V$, denoted $\dim(A; V)$, is the smallest real number $r$ such that there exists a real number $c$ with $f_{A;V}(n) \leq c n^r$ for all $n$. If no such real number exists, we set $\dim(A;V) = \infty$.
\end{defn}
As with the ordinary Gelfand--Kirillov dimensions, we have:
\begin{proposition}\label{prop:GKdimind}
    Assume $V$ and $W$ are two finite length subobjects of $A$ that generate it. Then $\dim(A;V) = \dim(A;W)$.
\end{proposition}
\begin{proof}
    It suffices to show that $\dim(A;V) \leq \dim(A;W)$. There exists $l > 0$ such that $W \subset A_{V, \leq l}$, so we have $f_{A; W}(n)  \leq f_{A; V}(nl)$, which in turn gives the required inequality.
\end{proof}
We set $\GKdim(A) = \dim(A;V)$ for any finite length generating subobject $V \subset A$; this is well-defined by the above proposition. We can similarly define $\GKdim$ for module objects for $A$ in $\cC$. The next result, along with the example that follows, allows us to compute $\GKdim$ for invariant subrings of finite groups when $p=0$. Interestingly, it is finite. 
\begin{lemma}
    Assume $A \to B$ is a map of algebra objects in $\cC$ such that $B$ is finitely generated as an $A$-module and the map splits as a map of $A$-modules. Then $\GKdim(A) = \GKdim(B)$.
\end{lemma}
\begin{proof}
This is similar to the proof of Proposition~\ref{prop:GKdimind} so we skip it.
\end{proof}

Recall that the ordinary Gelfand--Kirillov dimension of noncommutative tensor algebras is infinite. We now compute $\GKdim(T(W))$ as a tca.
\begin{example}
Let $W$ be an $n$-dimensional vector space. By Cauchy's formula, the $\GL$-algebra $\Sym(W \otimes \bV)$ has a filtration with associated graded pieces $\bS_{\lambda}(W) \otimes \bS_{\lambda}(\bV)$ with $\lambda$ varying over all partitions. Applying the Schur functor and using Lemma~\ref{lem:schurfacts}, we get that $T(W)$ has a filtration with associated graded pieces $\bS_{\lambda}(W) \otimes S^{\lambda}$. Therefore, we get
\[ f_{A; W}(N) = \sum_{|\lambda| \leq N} \dim(\bS_{\lambda}(W)) \len(S^{\lambda}).\] 
Recall that $\bS_{\lambda}(W)$ vanishes if $\lambda$ has $> \dim(W)$ nonzero parts. For the sake of simplicity, we now assume $\chark(k) = p = 0$ so that $\len(S^{\lambda}) = 1$ in the above sum.
We now estimate $f_{A; W}$: the number of partitions of $N$ with at most $n$ parts is $\approx c N^{n-1}$ with $c$ a constant for $N \gg n$; and the dimension of the Schur module $S_{\lambda}(W)$ with $|\lambda| \gg n$ and $\lambda_{n+1} = 0$ is $\approx d |\lambda|^{\frac{n(n-1)}{2}}$ where $d$ is another constant, as can be seen using Weyl's dimension formula. Therefore, $f_{A;W}(N) \approx \sum_{i=0}^N cd i^{n-1} i^{\frac{n(n-1)}{2}}$ which has dominant term $N^{\frac{n(n+1)}{2}}$. So $\GKdim(T(W)) = \binom{n+1}{2}$. We emphasize that this agrees with the computation of the Krull--Gabriel dimension of the module category of the tca $T(W)$ in \cite[\S~3.5]{ss19gl2}.
\end{example}

It would be interesting to also compute $\GKdim$ of $T(W)$ when $p > 0$. If $\dim(W) \in \{1, 2\}$, then $\GKdim$ is independent of characteristic, as $\len(S^{\lambda})$ for partitions with $\leq 2$ parts is at most $\approx\log(|\lambda|)$ (see references cited in Example~\ref{exm:SL2}). 

We note that $\GKdim$ does not provide anything meaningful for polynomial tcas generated in degree $\geq 2$. We show it is infinite for the polynomial algebra of the trivial representation $\triv_2$ of $\fS_2$.
\begin{example}\label{exm:symsym2}
The $\GL$-algebra $\Sym(\Sym^2(\bV))$ has a multiplicity-free filtration by Schur modules $\bS_{2\lambda}(\bV)$ \cite{bof91plethysm}, where $2\lambda$ is the partition $(2\lambda_1, 2\lambda_2, \ldots)$. Applying the Schur functor, we see that the tca $A = \Sym(\triv_2)$ has a filtration by the Specht modules $S_{2 \lambda}$. 
Let $p(n)$ be the partition function. Then in characteristic zero, we have $f_{A; \triv_2}(n) = \sum_{i\leq n} p(i)$, and $f_{A; \triv_2} \geq \sum_{i \leq n} p(i)$ since $\len(S_{\lambda}) \geq 1$. The partition function is subexponential so $\GKdim(A) = \infty$ in all characteristics. 
\end{example}
We finally compute $\GKdim$ for $T(k^2)^{\SL_2(k)}$ with $k$ algebraically closed.
\begin{example}\label{exm:SL2}
On the $\GL$-algebra side, the $\SL_2(k)$-invariant subring of $\Sym(\bV \oplus \bV)$ is isomorphic to the quotient of $\Sym(\lw^2(\bV))$ by the ideal generated by the Pl\"ucker relations. By \cite[Theorem~2.7]{bof91plethysm}, this quotient ring has a multiplicity-free filtration by the Schur modules $\bS_{(2^n)}(\bV)$. Applying the Schur functor and using Theorem~\ref{thm:structure} and Lemma~\ref{lem:schurfacts}(3), we see that $A \coloneqq T(k^2)^{\SL_2(k)}$ has a multiplicity-free filtration by the Specht modules $S_{(2^n)}$.
If $\chark(k) = p =0$, we immediately see that $\GKdim(A) = 1$. 
If instead $p > 0$, then we have $S^{\lambda} \cong S^{\lambda^t} \otimes \sgn$ where $\sgn$ is the sign representation of $\fS_{|\lambda|}$, so $\len(S_{(2^n)}) = \len(S_{(n, n)})$. 
By the work of James (see \cite[Theorem~2.5]{fay03specht}), the length of $S_{(n, n)}$ is $\approx {\lfloor \log_p(\frac{n+1}{2}) \rfloor}$. Combining the previous two facts, we get that $f_{A;A_2}(n+1) - f_{A; A_2}(n) \approx \log_p(n)$ from which we obtain $\GKdim(A) = 1$.
\end{example}
\bibliographystyle{alpha}
\bibliography{bibliography}
\end{document}